\documentclass[letterpaper, 10 pt, conference]{ieeeconf}

\IEEEoverridecommandlockouts                              
\newif\ifpdf
\ifx\pdfoutput\undefined
	\pdffalse
\else
	\pdfoutput=1
	\pdftrue
\fi
\ifpdf
	\usepackage{graphicx,graphics,latexsym,verbatim,epsfig}
	\DeclareGraphicsExtensions{.pdf,.jpg,.png}
\else
	\usepackage{graphicx}
	\DeclareGraphicsExtensions{.eps}
\fi

\makeatletter
\def\maxwidth{
  \ifdim\Gin@nat@width>\linewidth
    \linewidth
  \else
    \Gin@nat@width
  \fi
}
\makeatother

\usepackage{placeins}
\usepackage{amsmath}
\usepackage{amsfonts}
\usepackage{amssymb}
\usepackage{multirow}
\usepackage[noadjust]{cite}
\usepackage{fixltx2e}

\newtheorem{theorem}{\bf Theorem}

\newtheorem{lemma}{Lemma}

\newcommand{\refe}[1]{(\ref{#1})}

\title{\LARGE \bf
Rearranging trees for robust consensus
}

\author{George Forrest Young, Luca Scardovi and Naomi Ehrich Leonard
\thanks{This research was supported in part by AFOSR grant FA9550-07-1-0-0528 and ONR grant N00014-09-1-1074. }
 \thanks{G. F. Young and N. E. Leonard are with the Department of Mechanical and Aerospace Engineering,
        Princeton University, Princeton, NJ 08544, USA. L. Scardovi is with the Department of Electrical Engineering and Information Technology, Technical University of Munich, 80333 Munich, Germany. 
        {\tt\small gfyoung@princeton.edu},  {\tt\small scardovi@tum.de}, {\tt\small naomi@princeton.edu}}%
}

\begin{document}

\maketitle
\thispagestyle{empty}
\pagestyle{empty}

\begin{abstract}
In this paper, we use the $\mathcal{H}_2$ norm associated with a communication graph to characterize the robustness of consensus to noise. In particular, we restrict our attention to trees and by systematic attention to the effect of local changes in topology, we derive a partial ordering for undirected trees according to the $\mathcal{H}_2$ norm. Our approach for undirected trees provides a constructive method for deriving an ordering for directed trees. Further, our approach suggests a decentralized manner in which trees can be rearranged in order to improve their robustness.
\end{abstract}

\section{INTRODUCTION}\label{sec:intro}

The study of linear consensus problems has gained much attention in recent years \cite{OlfatiSaber2004, Moreau2005, Ren2005, Blondel2005}. This attention has stemmed, in part, from the wide range of applications of linear consensus, including collective decision-making \cite{Ren2005}, formation control \cite{Bamieh2008}, sensor fusion \cite{OlfatiSaber2005}, distributed computing \cite{Xiao2007} and the understanding of biological groups \cite{Sumpter2008}. In most of these applications, information passed between agents can be corrupted by noise. It is therefore necessary to understand and characterize the robustness of consensus when noise is present, with the goal of designing systems that can efficiently filter noise and remain close to consensus. For a linear system with additive white noise, a natural measure of robustness is the $\mathcal{H}_2$ norm \cite{Young2010}.

For the study of consensus, most of the important details of a multi-agent system are described by the communication graph. In fact, the properties of the Laplacian matrix of the graph are deeply related to the performance of the linear consensus protocol. In this way, the study of consensus often reduces to studying the underlying graph, and relating graph properties to the performance of the original system \cite{Wu2007, OlfatiSaber2004, Ren2005, Scardovi2009}.

Communication in a multi-agent system is likely to be of a directed nature, simply because each agent may treat the information it receives differently than its neighbors. Additionally, directed communication can arise when information is transferred through sensing or when agents have limited capabilities and may choose to only receive information from a subset of possible neighbors.

In many real systems, the graph between agents is not fixed, but may change over time depending on the behavior and decisions of individual agents \cite{OlfatiSaber2004}. Therefore, when searching for ways in which to impose effective graphs on multi-agent systems, it is highly advantageous to consider whether such graphs could be formed in a decentralized manner.

In this paper, we study the robustness of a particular family of graphs, namely trees, according to their $\mathcal{H}_2$ norms. We develop a partial ordering among trees that allows us to find a tree with minimal $\mathcal{H}_2$ norm, given certain constraints. Although most of this partial ordering has already been developed in the literature on Wiener indices \cite{Dobrynin2001, Dong2006, Deng2007, Wang2008a}, our methods of proof are new. In particular, we rely only on local changes in which one or more leaf nodes are moved from a single location in the tree to a new location. This approach provides insight into ways in which trees can be rearranged in a decentralized manner in order to improve their robustness. Additionally, our methods can be used to derive a similar ordering for directed trees that could not be found using the Wiener index literature.

This paper is organized as follows. In Section \ref{sec:prelim} we summarize notation. In Section \ref{sec:H2} we discuss the $\mathcal{H}_2$ norm in more detail. In Sections \ref{sec:resistance} and \ref{sec:indices} we discuss the relationship between the $\mathcal{H}_2$ norm and other graph indices. In Section \ref{sec:treeterm} we introduce a system of terminology to describe tree graphs, and in Section \ref{sec:manipulate} we derive our partial ordering. Finally, in Section \ref{sec:disc}, we discuss the potential for a decentralized algorithm to improve the $\mathcal{H}_2$ norm of a tree.

\section{PRELIMINARIES AND NOTATION}\label{sec:prelim}
The state of the system is given by $x = \left[ x_1, x_2, \ldots, x_N \right] \in \mathbb{R}^N$, where $x_i$ is the state of agent $i$. For each agent $i$ we define the set of neighbors, $\mathcal{N}_i$, to be the set of agents which supply information to agent $i$.

We call the state of the system a \emph{consensus} state when $x = \gamma 1_N$, where $1_N = [1,1,\ldots,1]^T \in \mathbb{R}^N$ and $\gamma \in \mathbb{R}$. Let $\Pi$ be the orthogonal projection matrix onto the subspace of $\mathbb{R}^N$ orthogonal to $1_N$. Thus $\Pi = I_N - \frac{1}{N}1_N 1_N^T$, where $I_N$ is the $N$-dimensional identity matrix. The system is in consensus if and only if $\Pi x = 0$.

We associate to the system a \emph{communication graph} $\mathcal{G} = \left( \mathcal{V}, \mathcal{E}, A \right)$, where $\mathcal{V} = \left\{ 1, 2, \ldots, N \right\}$ is the set of nodes, $\mathcal{E} \subseteq \mathcal{V}\times\mathcal{V}$ is the set of edges and $A \in \mathbb{R}^{N\times N}$ is a weighted adjacency matrix with nonnegative entries $a_{i,j}$. $A$ is defined such that $a_{i,j} > 0$ if and only if $\left( i,j \right) \in \mathcal{E}$. Every node in the graph corresponds to an agent in our system, while the graph contains edge $\left(i,j\right)$ when $j \in \mathcal{N}_i$. That is, every directed edge in $\mathcal{G}$ points from an agent receiving information to the agent supplying the information. Then $a_{i,j}$ is the weight given by agent $i$ to the information from agent $j$. Note that according to our definition, $\mathcal{G}$ will contain at most one edge between any ordered pair of nodes and will not contain any self-cycles (edges connecting a node to itself). 

An edge $(i,j) \in \mathcal{E}$ is said to be \emph{undirected} if $(j,i)$ is also in $\mathcal{E}$ and $a_{i,j} = a_{j,i}$. A graph is undirected if every edge is undirected, that is, if $A$ is symmetric. A graph that is not undirected is called directed. Two nodes are said to be \emph{adjacent} when there is an edge between them (i.e. when the corresponding agents communicate with each other).

The \emph{out-degree} of node $k$ is defined as $d_k^{out} = \sum_{j=1}^N{a_{k,j}}$. $\mathcal{G}$ has an associated \emph{Laplacian} matrix $L$, defined by $L = D - A$, where $D = \mathrm{diag}\left(d_1^{out}, d_2^{out}, \ldots, d_N^{out}\right)$ is the diagonal matrix of node out-degrees. The row sums of the Laplacian matrix are zero, that is $L 1_N = 0$. Thus $0$ is always an eigenvalue of $L$ with corresponding eigenvector $1_N$.  Furthermore, all eigenvalues of $L$ have non-negative real part (by Ger\v{s}gorin's Theorem). For an undirected graph, $L$ is symmetric, so in addition $1_N^T L = 0$ and all the eigenvalues of $L$ are real.

A \emph{path} in a  graph $\mathcal{G}$ is a (finite) sequence of nodes containing no repetitions and such that each node is a neighbor of the previous one. The length of a path is given by the sum of the weights on all edges traversed by the path. 

The graph $\mathcal{G}$ is \emph{connected} if it contains a globally reachable node $k$; i.e. there exists a node $k$ such that there is a path in $\mathcal{G}$ from $i$ to $k$ for every node $i$. It can be shown that $0$ will be a simple eigenvalue of $L$ if and only if $\mathcal{G}$ is connected \cite{Mohar1991a}. If an undirected graph is connected, there will be a path in $\mathcal{G}$ between every pair of nodes.

The \emph{distance}, $d_{i,j}$, between two nodes $i$ and $j$ in a graph is the shortest length of any path connecting the nodes. If no such path exists, the distance is infinite. The \emph{diameter}, $d$, of a graph is the maximum distance between all pairs of nodes in the graph. An undirected graph that is connected will have a finite diameter, while a graph that is not connected will have infinite diameter. A directed graph can be connected and have infinite diameter.

We use $\lambda_i$ to refer to the $i^\mathrm{th}$ eigenvalue of the Laplacian matrix, when arranged in ascending order by real part. Thus $\lambda_1 = 0$ for any Laplacian matrix, and $\mathrm{Re}\left\{\lambda_2\right\} > 0$ if and only if $\mathcal{G}$ is connected. For an undirected graph, $\lambda_2$ is referred to as the \emph{algebraic connectivity} \cite{Mohar1991a}.

A \emph{tree} on $N$ nodes is a connected undirected graph in which every pair of nodes is connected by a unique path. This implies that a tree contains exactly $N - 1$ undirected edges and that it contains no cycles (paths with positive length connecting a node to itself). A \emph{rooted} tree is a tree in which one particular node has been identified as the root (note that other than being called the root, there is nothing ``special'' about this node). A \emph{directed tree} is a connected graph containing exactly $N - 1$ directed edges. In a directed tree, the globally reachable node is identified as the root.

The floor function of a real number $x$, denoted $\lfloor x \rfloor$, is the largest integer that is less than or equal to $x$. The ceiling function of a real number $x$, denoted $\lceil x \rceil$, is the smallest integer that is greater than or equal to $x$.

\section{ROBUST NOISY CONSENSUS AND THE $\mathcal{H}_2$ NORM}\label{sec:H2}

There are many potential sources of noise in consensus dynamics, including communication errors, spurious measurements and time delays. In this paper we assume that every agent is independently affected by white noise of the same intensity. The resulting dynamics are
\begin{equation}\label{eqn:sysn}
\dot{x}(t) = -Lx(t) + \xi(t)
\end{equation}
with $x \in \mathbb{R}^N$ and where $\xi(t) \in \mathbb{R}^N$ is a random signal with $E[\xi(t)] = 0$, $E[\xi(t)\xi^T(\tau)] = \frac{\alpha}{2}I_N\delta(t-\tau)$ and $E[x(0)\xi^T(\tau)] = 0$. $\delta(t)$ is the Dirac delta function and $\alpha > 0$ is the intensity of the noise.

Since \refe{eqn:sysn} is only marginally stable in the noise-free case (corresponding to the fact that there is no ``preferred'' or ``correct'' value for the agents to agree upon), we only consider the dynamics on the subspace of $\mathbb{R}^N$ orthogonal to the subspace spanned by $1_N$. We let $Q \in \mathbb{R}^{(N-1)\times N}$ be a matrix with rows that form an orthonormal basis of this subspace. This is equivalent to requiring that
\begin{equation}\label{eqn:Q}
\begin{split}
Q 1_N = 0, &\\
Q Q^T = I_{N-1} &\text{ and } Q^T Q = I_N - \frac{1}{N}1_N 1_N^T = \Pi.
\end{split}
\end{equation}
Note that $L Q^T Q = L\left(I_N - \frac{1}{N}1_N 1_N^T \right) = L$, as $L 1_N = 0$. Next, we define $y \mathrel{\mathop:}= Qx$. Then $y = 0$ if and only if $x = \gamma 1_N, \gamma \in \mathbb{R}$.   A measure of the distance from consensus is the \emph{dispersion} of the system $\left|\left|y(t)\right|\right| = \left(y^T(t) y(t)\right)^{\frac{1}{2}}$. Note that the projection of \refe{eqn:sysn} onto the subspace spanned by $1_N$ will give the dynamics of the mean of $x$. These dynamics remain marginally stable (in the noise-free case), and perform a random walk with noise present.

Differentiating $y(t)$, we obtain
\begin{equation}\label{eqn:reducedsys}
\dot{y}(t) = -\bar{L}y(t) + Q\xi(t)
\end{equation}
where $\bar{L} = Q L Q^T$ is the \emph{reduced Laplacian} matrix.

Note that $\bar{L}$ is not unique, since we can compute it using any matrix $Q$ that satisfies \refe{eqn:Q}. However, if $Q$ and $Q^\prime$ both satisfy \refe{eqn:Q}, we can define $P \mathrel{\mathop:}= Q^{\prime}Q^T$. Then $Q^\prime = PQ$ and $P$ is orthogonal. Therefore, if $y^\prime(t) \mathrel{\mathop:}= Q^\prime x(t) = Py(t)$, we have that $y^{\prime T}(t)y^\prime(t) = y^{T}(t)P^{T}Py(t) = y^{T}(t)y(t)$ and thus the dispersion is invariant to the choice of $Q$.

In \cite{Young2010} we demonstrated that $\bar{L}$ has the same eigenvalues as $L$ but the zero eigenvalue, which implies that $-\bar{L}$ is Hurwitz precisely when the graph is connected. Thus, for a connected graph in the absence of noise, system \refe{eqn:reducedsys} will converge exponentially to zero with convergence speed given by $\mathrm{Re}\left\{\lambda_2\right\}$. Throughout the rest of this paper, we will assume that every graph is connected.

In the presence of noise, system \refe{eqn:reducedsys} will no longer converge to zero, but will remain in motion about zero. We therefore define the robustness of consensus to noise as the expected dispersion of the system in steady state. Note that this definition is analogous to the steady-state mean-square deviation used in \cite{Xiao2007}. It turns out that our measure of robustness corresponds to the $\mathcal{H}_2$ norm of system \refe{eqn:reducedsys}, with output equation $z(t) = I_{N-1}y(t)$. A similar approach, of measuring robustness using the $\mathcal{H}_2$ norm, has been taken in \cite{Bamieh2008, Zelazo2009}.

In \cite{Young2010} we proved that for a system with an undirected communication graph\footnote{In fact, the result in \cite{Young2010} is more general and extends to directed graphs with a normality condition on their Laplacian matrix.}, the $\mathcal{H}_2$ norm is given by
\begin{equation}\label{eqn:H2normal}
H = \left(\sum_{i=2}^{N}{\frac{1}{2\lambda_i}}\right)^{\frac{1}{2}}.
\end{equation}

In general, the $\mathcal{H}_2$ norm of a directed graph can be computed as $H = \left[\mathrm{tr}(\Sigma)\right]^{\frac{1}{2}}$, where $\Sigma$ is the solution to the Lyapunov equation \cite{Young2010}
\begin{equation}\label{eqn:lyap}
\bar{L}\Sigma + \Sigma \bar{L}^T = I.
\end{equation}

Since this $\mathcal{H}_2$ norm can be computed entirely from knowledge of the communication graph, in the rest of this paper we associate the $\mathcal{H}_2$ norm with the graph. Thus when we refer to the $\mathcal{H}_2$ norm of a graph, we mean the $\mathcal{H}_2$ norm of system \refe{eqn:reducedsys} (with output $z = y$) with $\bar{L}$ computed from the given graph.

In \cite{Young2010}, we compared a number of directed graphs with respect to their $\mathcal{H}_2$ norms. The main purpose of the present paper is to investigate the robustness associated to the class of undirected tree graphs, with the aim of extending this analysis to directed trees as well. Within this class, the star and path graphs have already been characterized in \cite{Young2010}. In this work we study other trees by investigating their order with respect to the $\mathcal{H}_2$ norm.

\section{EFFECTIVE RESISTANCE AS A MEASURE OF THE $\mathcal{H}_2$ NORM}\label{sec:resistance}

Although our formula in equation \refe{eqn:H2normal} allows us to compute the $\mathcal{H}_2$ norm for any undirected graph, it does not readily allow us to infer relationships between structural features of the graph and the $\mathcal{H}_2$ norm. However, the concept of the \emph{effective resistance}, or \emph{Kirchhoff index}, of a graph can help us in this respect.

The effective resistance of a graph can be related to the power dissipated by the graph when it is considered as a resistor network, and also to the expected commute time between any two nodes for a random walk with transition probabilities governed by the edge weights \cite{Ghosh2008}. It is related to the eigenvalues of the graph Laplacian \cite{Xiao2003} by the formula $\displaystyle K_f = N\sum_{i=2}^N{\frac{1}{\lambda_i}}$, leading to the relationship 
\begin{equation}\label{eqn:H2resist}
H = \left(\frac{K_f}{2N}\right)^{\frac{1}{2}}.
\end{equation}
We see, therefore, that the effective resistance is equivalent to the $\mathcal{H}_2$ norm in that one can be computed from the other, and for graphs with equal numbers of nodes, any ordering induced by one measure is the same as the ordering induced by the other.

Conceptually, the effective resistance results from considering a given graph as an electrical network, where every edge corresponds to a resistor with resistance given by the inverse of the edge weight. The resistance between two nodes in the graph is given by the resistance between those two points in the electrical network, and the effective resistance of the graph is given by the sum of the resistances between all pairs of nodes in the graph \cite{Xiao2003}.

One immediate result from this approach is that since adding resistors in parallel to, or decreasing resistances in, an existing network cannot increase the overall resistance, we see that adding edges to, or increasing edge weights in, an existing undirected graph will strictly decrease the effective resistance, and hence the $\mathcal{H}_2$ norm. Therefore, the $\mathcal{H}_2$ norm of a spanning tree of an undirected graph will provide an upper bound on the $\mathcal{H}_2$ norm of the original graph.

Although computing the effective resistance can be difficult for graphs that contain many cycles, it is very straightforward for trees. In a tree, there is precisely one path joining any pair of nodes, so the resistance between them is simply the sum of the resistances of each edge along that path. In a tree with unit weights on every edge, the resistance between two nodes is given by the distance between them \cite{Klein1993}. Hence, the effective resistance of a tree with unit edge weights is given by
\begin{equation}\label{eqn:treeresist}
K_f = \sum_{i < j}{r_{i,j}} = \sum_{i < j}{d_{i,j}}
\end{equation}

Although the concept of effective resistance does not apply to directed graphs, we can define an extension so that equation \refe{eqn:H2resist} applies to directed graphs as well. The resistance between two nodes of an undirected graph can be computed as \cite{Xiao2003}
\begin{equation}
r_{i,j} = (L^\dagger)_{i,i} + (L^\dagger)_{j,j} - 2(L^\dagger)_{i,j}
\end{equation}
where $L^\dagger$ is the Moore-Penrose pseudoinverse of $L$. For an undirected graph, we can explicitly compute $L^\dagger$ as $L^\dagger = Q^T\bar{L}^{-1}Q$. However, the solution to the Lyapunov equation \refe{eqn:lyap} for an undirected graph is $\Sigma = \frac{1}{2}\bar{L}^{-1}$. Therefore, if we let $X = 2Q^T\Sigma Q$, we can write
\begin{equation}\label{eqn:dirresist}
r_{i,j} = (X)_{i,i} + (X)_{j,j} - 2(X)_{i,j}
\end{equation}

Using equation \refe{eqn:dirresist}, we can compute ``directed resistances'' (and hence Kirchhoff indices) for directed graphs. Through this construction, we can show that equation \refe{eqn:H2resist} will hold for directed graphs as well. The advantage of using this approach to compute the $\mathcal{H}_2$ norm of a directed tree is that, like an undirected tree, the resistance between two nodes only depends on the paths between them. The proofs and results for directed trees will appear in a future publication.

Here we use the concept of effective resistance to determine a partial ordering of undirected trees with unit edge weights. The same ordering will apply to the set of trees with a given constant edge weight, as all resistances will be proportional to those in the corresponding tree with unit edge weights.

\section{THE $\mathcal{H}_2$ NORM AND OTHER GRAPH INDICES}\label{sec:indices}

In addition to the Kirchhoff index, many other ``topological'' indices of graphs have arisen out of the mathematical chemistry literature \cite{Rouvray1987}. One of the earliest to arise was the Wiener index, $W$ \cite{Rouvray1987}. The Wiener index for any (undirected) graph is defined as
\begin{equation}
W = \sum_{i < j}{d_{i,j}}.
\end{equation}
Thus, for trees (with unit edge weights), the Kirchhoff and Wiener indices are identical. However, the two indices differ for any graph that is not a tree. Therefore, while the results in Section \ref{sec:manipulate} apply equally to Wiener indices of trees, we choose to interpret them only in terms of the Kirchhoff index and $\mathcal{H}_2$ norm.

Much work has already been done on comparing trees based on their Wiener indices. It is already well-known that the Wiener index of a tree will fall between that of the star and that of the path \cite{Entringer1994, Dobrynin2001}. For trees with a fixed number of nodes, the 15 trees with smallest Wiener index and the 17 trees with largest Wiener index have been identified \cite{Dong2006, Deng2007}. Further, for trees with a fixed number of nodes and a fixed diameter, the tree with smallest Wiener index has been found \cite{Wang2008a}. Therefore, most of the main results in Section \ref{sec:manipulate} have already been derived. Our contribution includes new methods of proof that rely on local changes of topology and provide constructive means to order directed trees and derive decentralized strategies for improving robustness.

A different graph index, developed in the mathematical literature, is the maximum eigenvalue of the adjacency matrix $A$ \cite{Simic2007}. Simi\'{c} and Zhou developed a partial ordering of trees with fixed diameter according to this index in \cite{Simic2007}.   Their work has motivated the approach taken in this paper. However, the ordering that we derive in this paper induced by the $\mathcal{H}_2$ norm differs from that induced by the maximum eigenvalue of $A$. While the tree with the smallest $\mathcal{H}_2$ norm for a given diameter (see Theorem \ref{thm:P}) corresponds to the tree with the largest maximum eigenvalue of $A$, the rest of the ordering in \cite{Simic2007} does not always hold in our case. For example, Theorem 6 and Theorem 8 in \cite{Simic2007} demonstrate that the trees we call $P_{N,d,i}$ (see Section \ref{sec:treeterm}) make up at least the first $\left\lfloor\frac{d}{2}\right\rfloor$ trees in a complete ordering. These results do not hold true for the $\mathcal{H}_2$ norm. Furthermore, we extend our comparisons to trees with different diameters.

The algebraic connectivity can also be used as an index to order trees \cite{Grone1990, Yuan2008}. Again, this ordering bears similarities to ours, but the complete ordering is not the same.

\section{A SYSTEM OF TERMINOLOGY FOR TREES}\label{sec:treeterm}

We first introduce a system of terminology relating to trees. Much of our terminology corresponds to that in \cite{Simic2007} and earlier papers, although we also introduce some additional terms. $\mathcal{T}_{N,d}$ is the set of all trees containing $N$ nodes and with diameter $d$. For $N \geq 3$, a tree must have $d \geq 2$, and $\mathcal{T}_{N,2}$ contains only one tree. This tree is referred to as a star, and is denoted by $K_{1,N-1}$ (as it contains one ``central'' node which is adjacent to the remaining $N-1$ nodes). For all positive $N$, the maximum diameter of a tree is $N-1$, and $\mathcal{T}_{N,N-1}$ contains only one tree. This tree is referred to as a path, and is denoted by $P_N$.

A \emph{leaf} (or pendant) node is a node with degree $1$. A \emph{bouquet} is a non-empty set of leaf nodes that are all adjacent to the same node. A node which is not a leaf is called an \emph{internal} node.

A \emph{caterpillar} is a tree for which the removal of all leaf nodes would leave a path. The set of all caterpillars with $N$ nodes and diameter $d$ is denoted by $\mathcal{C}_{N,d}$ (see Figure \ref{fig:caterpillar}). Any caterpillar in $\mathcal{C}_{N,d}$ contains a path of length $d$, with all other nodes adjacent to internal nodes of this path. In particular, we refer to the caterpillar that contains a single bouquet attached to the $i^\mathrm{th}$ internal node along this path $P_{N,d,i}$ (see Figure \ref{fig:Pndi}). To avoid ambiguity, we require  $1 \leq i \leq \lfloor \frac{d}{2}\rfloor$. The tree formed from $P_{N-1,d,\lfloor\frac{d}{2}\rfloor}$ by attaching an additional node to one of the leaves in the central bouquet is denoted by $N_{N,d}$ (see Figure \ref{fig:Nnd}). Note that in Figures \ref{fig:Pndi} and \ref{fig:Nnd}, we are only numbering the internal nodes in the longest path, since attaching nodes to either end of this path would create a longer path.

\begin{figure}[thb]
\centering
\includegraphics[width=0.63\maxwidth]{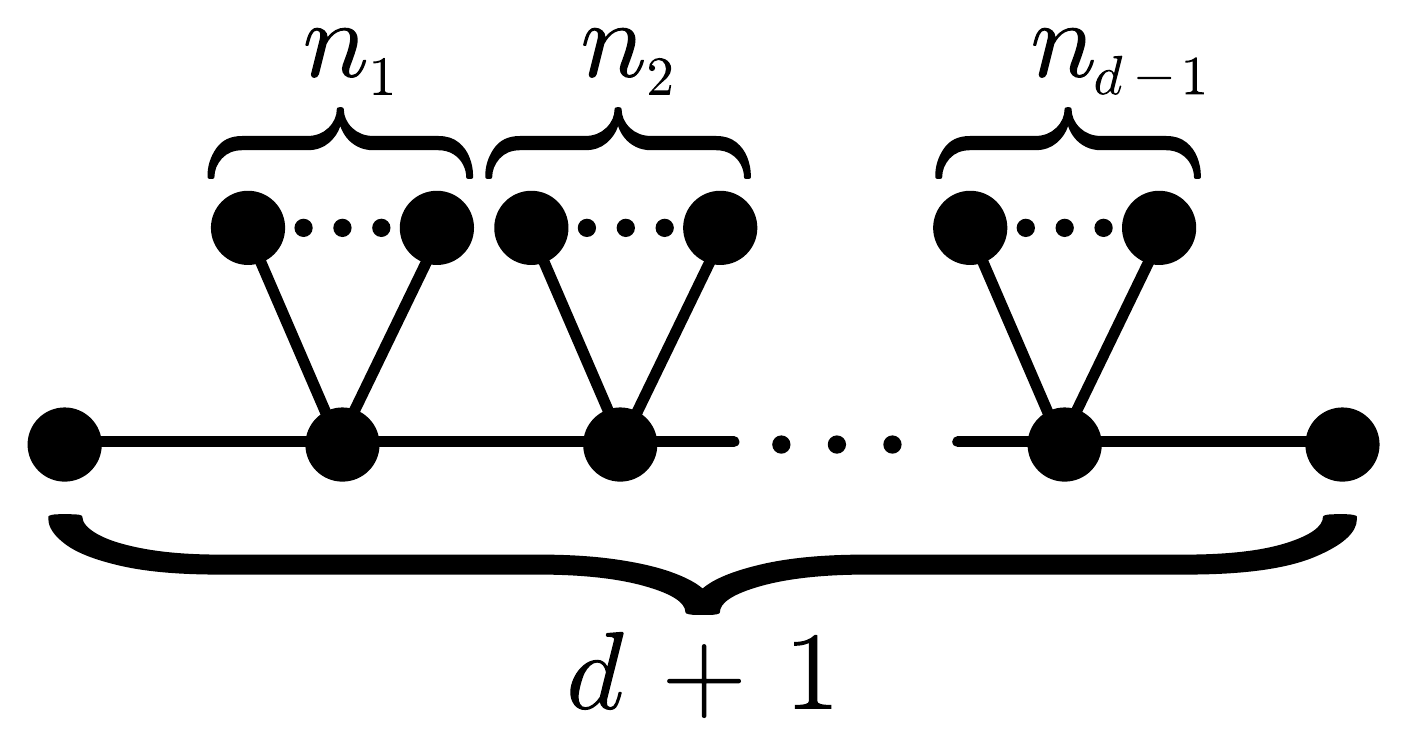}
\caption{General form of a caterpillar in $\mathcal{C}_{N,d}$, with $n_j \geq 0$ additional leaf nodes attached to each internal node $j$ in the path of length $d$.}
\vspace{-0.1 cm}
\label{fig:caterpillar}
\end{figure}

\begin{figure}[thb]
\centering
\includegraphics[width=0.7\maxwidth]{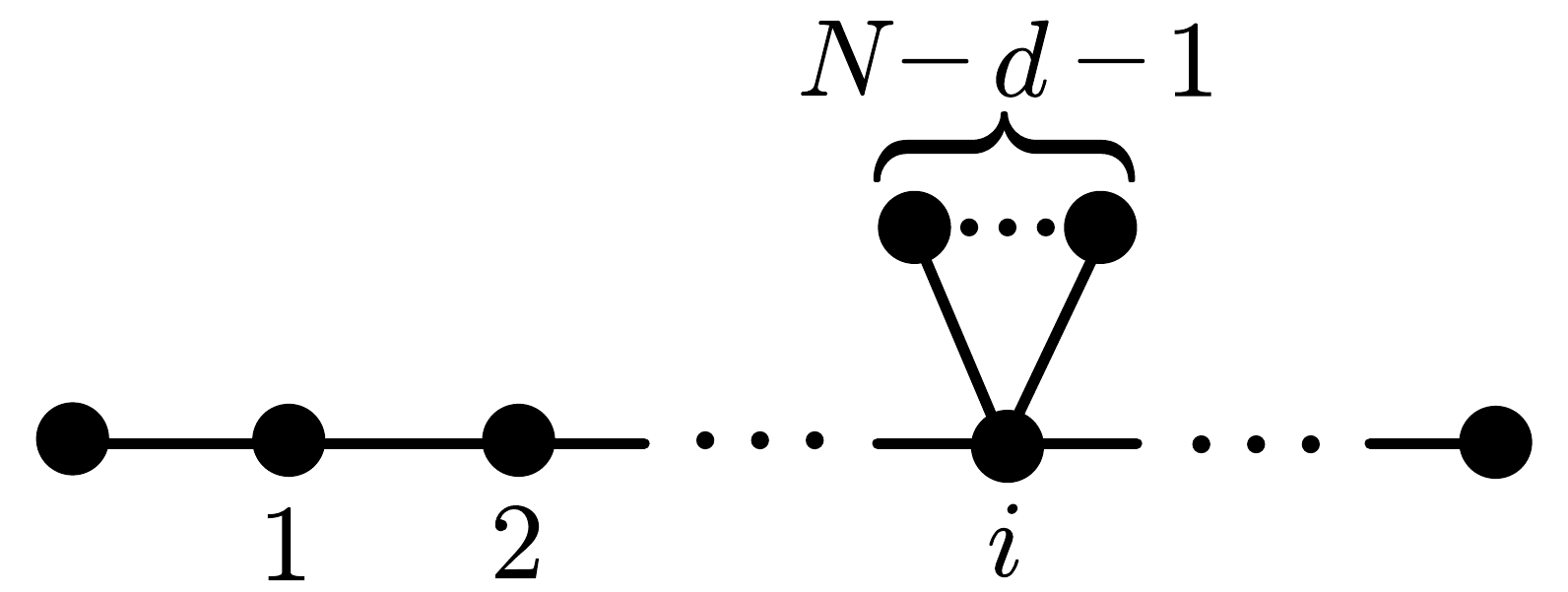}
\caption{The caterpillar $P_{N,d,i}$, a path of length $d$ with a bouquet containing $N-d-1$ leaf nodes attached to the $i^\mathrm{th}$ internal node on the path.}
\vspace{-0.1 cm}
\label{fig:Pndi}
\end{figure}

\begin{figure}[thb]
\centering
\includegraphics[width=0.7\maxwidth]{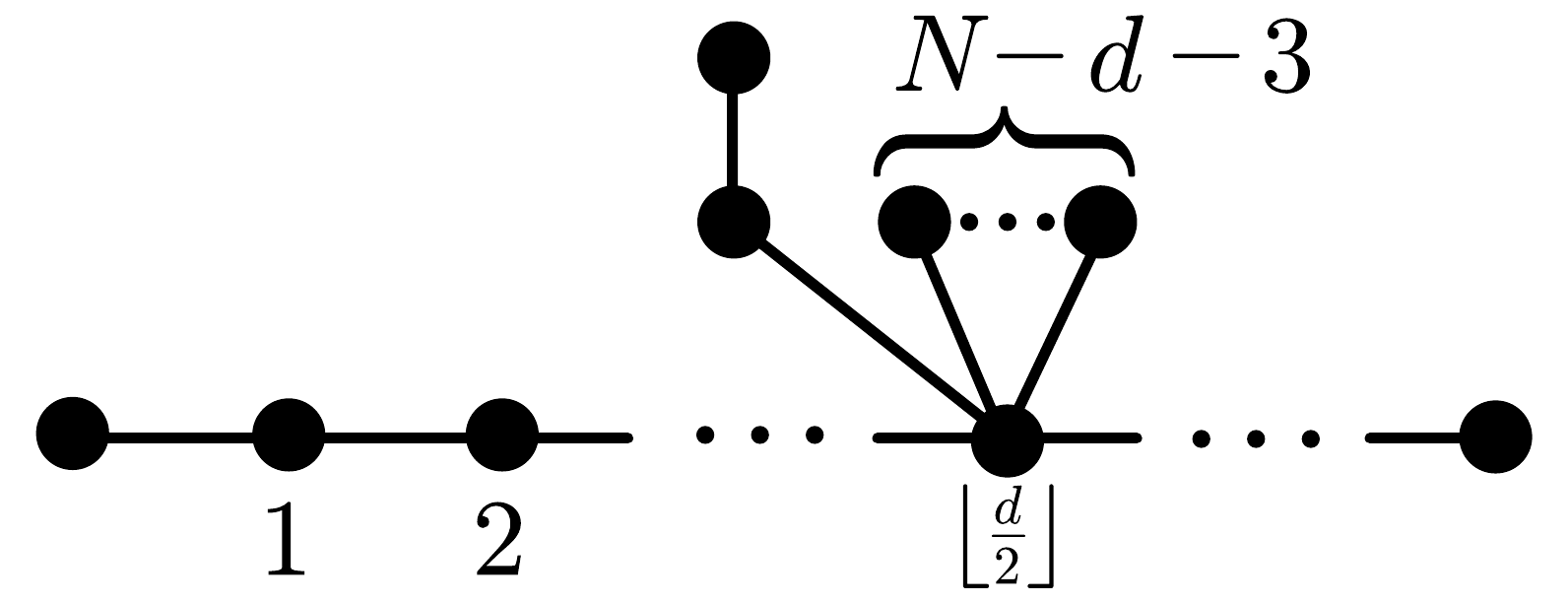}
\caption{The tree $N_{N,d}$, formed from $P_{N-1,d,\lfloor\frac{d}{2}\rfloor}$ by attaching an additional node to one of the leaves in the central bouquet. Note that $N-d-3$ could be $0$.}
\vspace{-0.1 cm}
\label{fig:Nnd}
\end{figure}

The \emph{double palm tree} (also referred to as a dumbbell in \cite{Dobrynin2001}) is a caterpillar with two bouquets, one at each end of the path (see Figure \ref{fig:double}). We use $D_{N,p,q}$ to denote the double palm tree on $N$ nodes, with bouquets of sizes $p$ and $q$. Note that this requires $p + q \leq N-2$ and that $D_{N,p,q} = D_{N,q,p}$. If we take a rooted tree $T$ (with root $r$) and separately attach two paths containing $l$ and $k$ nodes to the root, we call the resulting tree a \emph{vine} and denote it by $T^r_{l,k}$ (see Figure \ref{fig:Trlk}). Note that $T^r_{l,k} = T^r_{k,l}$.

\begin{figure}[thb]
\centering
\includegraphics[width=0.7\maxwidth]{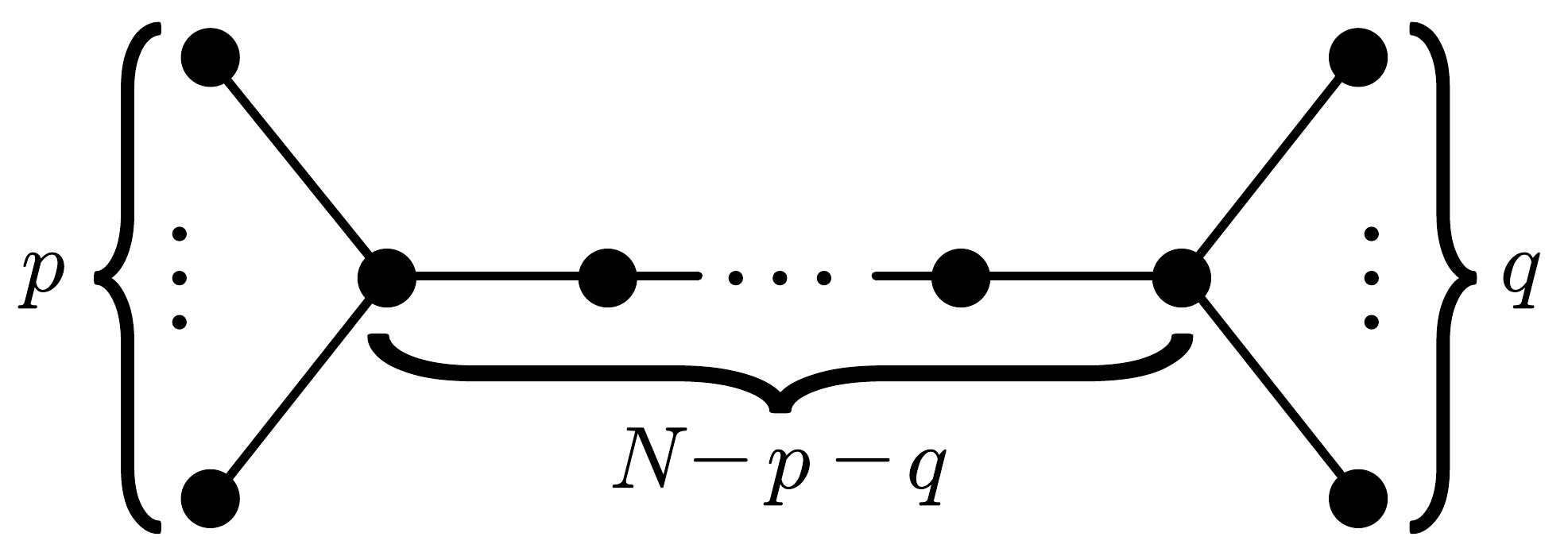}
\caption{Double palm tree $D_{N,p,q}$, with bouquets of sizes $p$ and $q$ at each end of a path.}
\vspace{-0.1 cm}
\label{fig:double}
\end{figure}

\begin{figure}[thb]
\centering
\includegraphics[width=0.7\maxwidth]{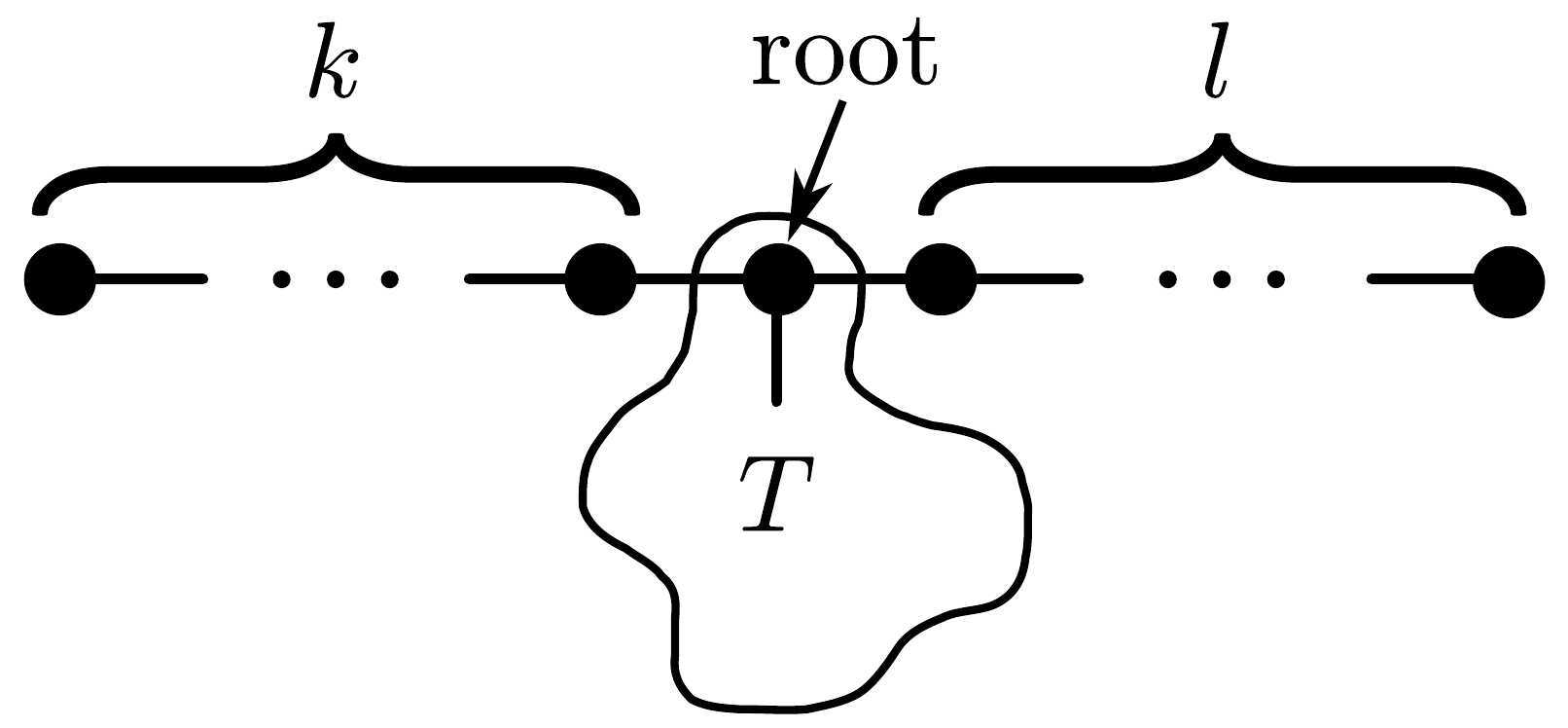}
\caption{The vine $T^r_{l,k}$, formed from a rooted tree $T$ by separately connecting paths containing $l$ and $k$ nodes to the root.}
\vspace{-0.1 cm}
\label{fig:Trlk}
\end{figure}

\section{MANIPULATIONS TO REDUCE THE EFFECTIVE RESISTANCE OF TREES}\label{sec:manipulate}

We can now start to describe a partial ordering on trees based on their $\mathcal{H}_2$ norms (i.e. their effective resistances). In this section, every tree is assumed to have a unit weight on every edge. First, we determine the effect of moving a leaf from one end of a double palm tree to the other, and use this to derive a complete ordering of all trees in $\mathcal{T}_{N,3}$ (Theorem \ref{thm:d3}). Second, we consider moving a leaf from one end of a vine to the other, and use this to prove that the path has the largest $\mathcal{H}_2$ norm of any tree with $N$ nodes (Theorem \ref{thm:path}), and to derive a complete ordering of $\mathcal{T}_{N,N-2}$ (Theorem \ref{thm:dN-2}). Finally, by moving all (or almost all) nodes in a bouquet to an adjacent node, we show that $P_{N,d,\lfloor\frac{d}{2}\rfloor}$ has the smallest $\mathcal{H}_2$ norm of any tree with diameter $d$ (Theorem \ref{thm:P}) and that the star has the smallest $\mathcal{H}_2$ norm of any tree with $N$ nodes (Theorem \ref{thm:d-1}). From Theorem \ref{thm:d-1} we also conclude that for any tree that is not a star, we can find a tree of smaller diameter with a smaller $\mathcal{H}_2$ norm.

\subsection{Double Palm Trees}\label{subsec:DNpq}
We begin our partial ordering by showing that the $\mathcal{H}_2$ norm of a double palm tree is reduced when we move a single node from the smaller bouquet to the larger one.

\begin{lemma}\label{lem:DNpq}
Let $1 < p \leq q$ and $p + q \leq N - 2$. Then $\mathcal{H}_2\left(D_{N,p,q}\right) > \mathcal{H}_2\left(D_{N,p-1,q+1}\right)$.
\end{lemma}

\begin{proof}
In $D_{N,p,q}$, let us label one of the nodes in the bouquet of size $p$ as node $1$. The remaining nodes are labelled $2$ through $N$. To form $D_{N,p-1,q+1}$, we take node $1$ and move it to the other bouquet. Since all other nodes remain unchanged, we can use equation \refe{eqn:treeresist} to write
\begin{align}\label{eqn:doublepalm}
K_f\left(D_{N,p,q}\right)& - K_f\left(D_{N,p-1,q+1}\right) = \nonumber \\
&\,\left(\sum_{j=2}^N{d_{1,j}}\right)_{D_{N,p,q}} - \left(\sum_{j=2}^N{d_{1,j}}\right)_{D_{N,p-1,q+1}}
\end{align}

Now, in $D_{N,p,q}$, the path length between node $1$ and any of the remaining $p-1$ nodes in the bouquet of size $p$ is $2$. Similarly, the path length between node $1$ and any node in the bouquet of size $q$ is $N - p - q + 1$. Finally, the path lengths between node $1$ and the internal nodes take on each integer value from $1$ to $N - p - q$.

Conversely, in $D_{N,p-1,q+1}$, the path length between node $1$ and any of the nodes in the bouquet of size $p-1$ is $N - p - q + 1$. The path length between node $1$ and any of the remaining $q$ nodes in the bouquet of size $q+1$ is $2$. Again, the path lengths between node $1$ and the internal nodes take on all integer values from $1$ to $N - p - q$.

Thus,
\begin{align}
\left(\sum_{j=2}^N{d_{1,j}}\right)_{D_{N,p,q}} &- \left(\sum_{j=2}^N{d_{1,j}}\right)_{D_{N,p-1,q+1}} =\nonumber \\
&\, (N-p-q-1)(q-p+1).
\end{align}
And this is positive since by our assumptions, $N - p - q - 1 \geq 1$ and $q - p + 1 \geq 1$. Therefore, by equation \refe{eqn:doublepalm}, $K_f\left(D_{N,p,q}\right) - K_f\left(D_{N,p-1,q+1}\right) > 0$, and so $K_f\left(D_{N,p,q}\right) > K_f\left(D_{N,p-1,q+1}\right)$. Hence, by equation \refe{eqn:H2resist},
\begin{equation}
\mathcal{H}_2\left(D_{N,p,q}\right) > \mathcal{H}_2\left(D_{N,p-1,q+1}\right)
\end{equation}
\end{proof}

Although Lemma \ref{lem:DNpq} applies to double palm trees with any diameter, we can apply it to trees with $d = 3$ in order to prove our first main result.

\begin{theorem}\label{thm:d3}
For $N \geq 4$, we have a complete ordering of $\mathcal{T}_{N,3}$, namely $\mathcal{H}_2\left(D_{N,1,N-3}\right) < \mathcal{H}_2\left(D_{N,2,N-4}\right) < \ldots < \mathcal{H}_2\left(D_{N,\lfloor\frac{N-2}{2}\rfloor,\lceil\frac{N-2}{2}\rceil}\right)$.
\end{theorem}

\begin{proof}
Any tree with $d = 3$ must have a longest path of length $3$. Any additional nodes in the tree must be connected through some path to one of the two internal nodes on this longest path (since connecting them to either end of the path would create a longer path). In addition, any node adjacent to one of the internal nodes of the longest path forms a path of length $3$ with the node at the far end of the path. Hence all such nodes must be leaves and so every tree with $d = 3$ is a double palm tree. The ordering follows from Lemma \ref{lem:DNpq}.
\end{proof}

\subsection{Vines}\label{subsec:vines}

We now switch our attention to finding an ordering of trees with the largest possible values of $d$. Our next lemma applies to trees of any diameter, but again we can specialize it to give the results we need.

\begin{lemma}\label{lem:Trlk}
Let $T$ be a tree containing more than one node and with a root $r$, and let $l,k$ be any positive integers such that $1 \leq l \leq k$. Then $\mathcal{H}_2\left(T^r_{l,k}\right) < \mathcal{H}_2\left(T^r_{l-1,k+1}\right)$.
\end{lemma}

\begin{proof}
Let the total number of nodes in $T^r_{l,k}$ be $N$ (so $N$ equals $k + l$ plus the number of nodes in $T$), and label the leaf at the end of the path containing $l$ nodes as $1$. Label the remaining nodes in the two paths as $2$ through $l+k$, and label the root of $T$ as $l+k+1$. The remaining nodes in $T$ are labelled as $l + k + 2$ through $N$. To form $T^r_{l-1,k+1}$, we take node $1$ and move it to the end of the other path. Since all other nodes remain unchanged, we can use equation \refe{eqn:treeresist} to write
\begin{align}\label{eqn:doublepath}
K_f&\left(T^r_{l-1,k+1}\right) - K_f\left(T^r_{l,k}\right) = \nonumber \\
&\,\left(\sum_{j=2}^N{d_{1,j}}\right)_{T^r_{l-1,k+1}} - \left(\sum_{j=2}^N{d_{1,j}}\right)_{T^r_{l,k}}.
\end{align}

Now, in both $T^r_{l,k}$ and $T^r_{l-1,k+1}$, the path lengths between node $1$ and all nodes along the paths (including the root of $T$) take on each integer value between $1$ and $l + k$. Hence the sum of these path lengths does not change between the two trees. Furthermore, since the root of $T$ lies on every path between node $1$ and any other node in $T$, we can write
\begin{equation}
d_{1,j} = d_{1,l+k+1} + d_{l+k+1,j}, \; j \geq l+k+2.
\end{equation}
Therefore, for $T^r_{l,k}$,
\[
\sum_{j=l+k+2}^N{d_{1,j}} = (N-l-k-1)l + \sum_{j=l+k+2}^N{d_{l+k+1,j}}
\]
and for $T^r_{l-1,k+1}$,
\[
\sum_{j=l+k+2}^N{d_{1,j}} = (N-l-k-1)(k+1) + \sum_{j=l+k+2}^N{d_{l+k+1,j}},
\]
where every term in the second summation is the same for both trees. Thus
\begin{align}
\left(\sum_{j=2}^N{d_{1,j}}\right)_{T^r_{l-1,k+1}} - &\left(\sum_{j=2}^N{d_{1,j}}\right)_{T^r_{l,k}} = \nonumber \\
&(N-l-k-1)(k-l+1),
\end{align}
which is positive by our assumptions. Therefore, by equations \refe{eqn:doublepath} and \refe{eqn:H2resist}, we have
\begin{equation}
\mathcal{H}_2\left(T^r_{l,k}\right) < \mathcal{H}_2\left(T^r_{l-1,k+1}\right).
\end{equation}
\end{proof}

The first consequence of Lemma \ref{lem:Trlk} is that the tree with largest $d$ (i.e. $d=N-1$) also has the largest $\mathcal{H}_2$ norm.

\begin{theorem}\label{thm:path}
The path $P_N$ has the largest $\mathcal{H}_2$ norm of any tree with $N$ nodes.
\end{theorem}

\begin{proof}
Any tree $T_1$ which is not a path will contain a node with degree greater than $2$. We can locate one such node that has two paths (each with fewer than $N$ nodes) attached. Let $T$ be the tree formed by removing these two paths from $T_1$, and let this node be the root of $T$. Then $T_1 = T^r_{l,k}$, and by Lemma \ref{lem:Trlk} we can find a tree with larger $\mathcal{H}_2$ norm.
\end{proof}

We can also use Lemma \ref{lem:Trlk} to derive an ordering of those trees with $d$ one less than its maximum value (i.e. $d = N-2$).

\begin{theorem}\label{thm:dN-2}
For $N \geq 4$, we have a complete ordering of $\mathcal{T}_{N,N-2}$, namely $\mathcal{H}_2\left(P_{N,N-2,\lfloor\frac{N-2}{2}\rfloor}\right) < \mathcal{H}_2\left(P_{N,N-2,\lfloor\frac{N-2}{2}\rfloor-1}\right) < \ldots < \mathcal{H}_2\left(P_{N,N-2,1}\right)$.
\end{theorem}

\begin{proof}
Every tree in $\mathcal{T}_{N,N-2}$ must contain a path of length $N-2$ (which contains $N-1$ nodes), and one additional node. This node must be adjacent to an internal node of the path, since otherwise we would have a path of length $N-1$. Thus every tree in $\mathcal{T}_{N,N-2}$ is of the form $P_{N,N-2,i}$, for some $1 \leq i \leq \lfloor\frac{N-2}{2}\rfloor$. Now, $P_{N,N-2,i} = T^r_{i,N-i-2}$, with $T$ a path containing $2$ nodes (and one identified as the root). Suppose that $i < \lfloor\frac{N-2}{2}\rfloor$. Then $i < N-i-2$, and so by Lemma \ref{lem:Trlk}, $\mathcal{H}_2\left(P_{N,N-2,i}\right)  = \mathcal{H}_2\left(T^r_{i,N-2-i}\right) > \mathcal{H}_2\left(T^r_{i+1,N-3-i}\right) = \mathcal{H}_2\left(P_{N,N-2,i+1}\right)$.
\end{proof}

Each tree in $\mathcal{T}_{N,N-2}$ consists of a path of length $N-2$ with one leaf attached to an internal node.  Theorem~\ref{thm:dN-2} ensures that the $\mathcal{H}_2$ norm is smallest when this internal node is at the center of the path.

\subsection{Caterpillars}\label{subsec:caterpillars}

We now have complete orderings for $\mathcal{T}_{N,2}$ (trivial, since $\mathcal{T}_{N,2}$ contains only the star), $\mathcal{T}_{N,3}$ (by Theorem \ref{thm:d3}), $\mathcal{T}_{N,N-2}$ (by Theorem \ref{thm:dN-2}) and $\mathcal{T}_{N,N-1}$ (trivial, since $\mathcal{T}_{N,N-1}$ contains only the path). We next consider the remaining families of trees with $4 \leq d \leq N-3$ (and hence, $N \geq 7$).

Rather than deriving complete orderings, the main goal of the next two lemmas is to find the tree in $\mathcal{T}_{N,d}$ with lowest $\mathcal{H}_2$ norm. However, we use two separate steps to attain our result as this provides greater insight into the ordering amongst the remaining trees in $\mathcal{T}_{N,d}$. Lemma \ref{lem:PvN} then allows us to combine the results to prove (Theorem \ref{thm:P}) that among trees of diameter $d$, the one with lowest $\mathcal{H}_2$ norm is the caterpillar with one bouquet at the central node $P_{N,d,\lfloor\frac{d}{2}\rfloor}$.  Theorem~\ref{thm:d-1} provides a comparison of trees with different diameter.

\begin{lemma}\label{lem:cat}
Suppose $N \geq 7$ and $4 \leq d \leq N-3$. If $T \in \mathcal{C}_{N,d}$, then $\mathcal{H}_2\left(T\right) \geq \mathcal{H}_2\left(P_{N,d,\lfloor\frac{d}{2}\rfloor}\right)$, with equality if and only if $T = P_{N,d,\lfloor\frac{d}{2}\rfloor}$. 
\end{lemma}

\begin{proof}
Since $d \leq N-3$ and $T \in \mathcal{C}_{N,d}$, a longest path in $T$ contains $N - d - 1 \geq 2$ leaves attached to internal nodes (other than the two leaves in the longest path). Suppose that $P_T$ is a longest path in $T$. For the rest of this proof, when we refer to leaf nodes and bouquets, we mean those leaves not on $P_T$, and bouquets made up of these leaves.

Suppose $T$ contains a single bouquet. Thus $T = P_{N,d,i}$ for some $1 \leq i \leq \lfloor\frac{d}{2}\rfloor$. If $i \neq \lfloor\frac{d}{2}\rfloor$, then by Lemma \ref{lem:Trlk}, $\mathcal{H}_2\left(P_{N,d,i}\right) > \mathcal{H}_2\left(P_{N,d,i+1}\right)$. Therefore, the result holds.

Suppose $T$ contains multiple bouquets. Locate a bouquet furthest from the center of $P_T$, and move every leaf in this bouquet one node further from the closest end of $P_T$. Call this new tree $T'$, and label the nodes that were moved $1$ through $n$. Then between $T$ and $T'$, the path lengths between each of these leaves and any other leaf all decrease by $1$. The path lengths between each of these leaves and $\leq \lfloor\frac{d+1}{2}\rfloor$ nodes on $P_T$ all increase by $1$, and the path lengths between each of these leaves and $\geq \lfloor\frac{d+1}{2}\rfloor$ nodes on $P_T$ all decrease by $1$. Thus the sum of the path lengths in $T'$ is less than the sum in $T$, and so by equations \refe{eqn:treeresist} and \refe{eqn:H2resist}, $\mathcal{H}_2\left(T'\right) < \mathcal{H}_2\left(T\right)$.

If $T'$ contains multiple bouquets, we can repeat this procedure until we obtain a caterpillar with a single bouquet. Then by the above argument, the result holds.
\end{proof}

\begin{lemma}\label{lem:noncat}
Suppose that $N \geq 7$ and $4 \leq d \leq N-3$. Let $T$ be a tree in $\mathcal{T}_{N,d}\setminus\mathcal{C}_{N,d}$. Then $\mathcal{H}_2\left(T\right) \geq \mathcal{H}_2\left(N_{N,d}\right)$, with equality if and only if $T = N_{N,d}$.
\end{lemma}

\begin{proof}
Let $P_T$ be a longest path in $T$ (of length $d$), and let $m$ be the number of nodes with distances to $P_T$ greater than $1$ (the distance between a node and $P_T$ is the shortest distance between that node and any node on the path).

If $m > 1$, locate a bouquet with the greatest distance from $P_T$ and label the leaves in this bouquet $1$ through $n$. Suppose that either the distance between this bouquet and $P_T$ is greater than $2$, or the distance is $2$ and another bouquet exists at a distance $2$ from $P_T$. Let $T'$ be the tree formed by moving all leaves in this bouquet one node closer to $P_T$. By our assumptions, $T' \in \mathcal{T}_{N,d}\setminus\mathcal{C}_{N,d}$. Then the path lengths between each of these nodes and the node they were previously adjacent to increase by $1$. Conversely, the path lengths between each of these nodes and every other node in $T$ decrease by $1$. Since there must be at least $d + 2 \geq 6$ of these other nodes, the sum of all path lengths in $T'$ is smaller than the sum of all path lengths in $T$. Thus, by equations \refe{eqn:treeresist} and \refe{eqn:H2resist}, $\mathcal{H}_2\left(T'\right) < \mathcal{H}_2\left(T\right)$.

If the bouquet we found has a distance of $2$ to $P_T$, and is the only such bouquet, form $T'$ by moving leaves $1$ through $n - 1$ one node closer to $P_T$. Then $T' \in \mathcal{T}_{N,d}\setminus\mathcal{C}_{N,d}$. The path lengths between each node we moved and the node they were previously attached to all increase by $1$. In addition, the path lengths between each node we moved and node $n$ all increase by $1$. However, the path lengths between each node we moved and the remaining $\geq d+1 \geq 5$ nodes in the tree all decrease by $1$. Thus, by equations \refe{eqn:treeresist} and \refe{eqn:H2resist}, $\mathcal{H}_2\left(T'\right) < \mathcal{H}_2\left(T\right)$.

If $m=1$, then $T$ contains a single node at a distance $2$ from $P_T$, and all other nodes in $T$ are either on $P_T$ or adjacent to nodes on $P_T$. Locate a node on $P_T$ with additional nodes attached that is furthest from the center of the path. Label all nodes attached to this node $1$ through $n$ (including the node at distance $2$ from $P_T$ if it is connected to $P_T$ through this node), and label this node $n+1$. If $n+1$ is not the $\lfloor\frac{d}{2}\rfloor^\mathrm{th}$ internal node on the path (i.e. if $T \neq N_{N,d}$), then form $T'$ by moving all nodes not on $P_T$ that are adjacent to $n+1$ (including the node at distance $2$, if present) one node further from the closest end of $P_T$. Then $T' \in \mathcal{T}_{N,d}\setminus\mathcal{C}_{N,d}$. Furthermore, the distances between every node originally adjacent to $n+1$ and all other nodes not on $P_T$ decrease by $1$. The path lengths between each of these nodes and $\leq \lfloor\frac{d}{2}\rfloor$ nodes on $P_T$ all increase by $1$, and the path lengths between each of these nodes and $\geq \lceil\frac{d}{2}\rceil+1$ nodes on $P_T$ all decrease by $1$. Thus the sum of the path lengths in $T'$ is less than the sum in $T$, and so by equations \refe{eqn:treeresist} and \refe{eqn:H2resist}, $\mathcal{H}_2\left(T'\right) < \mathcal{H}_2\left(T\right)$.

Hence for every tree in $\mathcal{T}_{N,d}\setminus\mathcal{C}_{N,d}$ other than $N_{N,d}$, there exists another tree in $\mathcal{T}_{N,d}\setminus\mathcal{C}_{N,d}$ with strictly smaller $\mathcal{H}_2$ norm.
\end{proof}

\addtolength{\textheight}{-7.4cm} 

\begin{lemma}\label{lem:PvN}
Suppose that $N \geq 7$ and $4 \leq d \leq N-3$. Then $\mathcal{H}_2\left(P_{N,d,\lfloor\frac{d}{2}\rfloor}\right) < \mathcal{H}_2\left(N_{N,d}\right)$.
\end{lemma}

\begin{proof}
Label the node in $N_{N,d}$ that is a distance $2$ from the lost path as node $1$, and label the node it is adjacent to as node $2$. Then we can form $P_{N,d,\lfloor\frac{d}{2}\rfloor}$ from $N_{N,d}$ by moving node $1$ one node closer to the longest path. Then $d_{1,j}$ decreases by $1$ for $j = 3, \dotsc,N$, and $d_{1,2}$ increases by $1$. Since $N \geq 7$, the sum of all path lengths in $P_{N,d,\lfloor\frac{d}{2}\rfloor}$ is less than in $N_{N,d}$. Thus, by equations \refe{eqn:treeresist} and \refe{eqn:H2resist}, $\mathcal{H}_2\left(P_{N,d,\lfloor\frac{d}{2}\rfloor}\right) < \mathcal{H}_2\left(N_{N,d}\right)$.
\end{proof}

Now, we have enough to determine the tree in $\mathcal{T}_{N,d}$ with smallest $\mathcal{H}_2$ norm. 
\begin{theorem}\label{thm:P}
Let $N \geq 4$ and $2 \leq d \leq N-2$. The tree in $\mathcal{T}_{N,d}$ with smallest $\mathcal{H}_2$ norm is $P_{N,d,\lfloor\frac{d}{2}\rfloor}$.
\end{theorem}

\begin{proof}
For $d=2$, $\mathcal{T}_{N,d}$ only contains $K_{1,N-1}$, which is the same as $P_{N,2,1}$. For $d=3$, the result follows from Theorem \ref{thm:d3} since $D_{N,1,N-3} = P_{N,3,1}$. For $4 \leq d \leq N-3$, this is a simple consequence of Lemmas \ref{lem:cat}, \ref{lem:noncat} and \ref{lem:PvN}. For $d = N-2$, the result follows from Theorem \ref{thm:dN-2}.
\end{proof}

Finally, we can combine several of our earlier results to obtain a basic comparison between trees of different diameters. 

\begin{theorem}\label{thm:d-1}
Let $3 \leq d \leq N-1$. For any tree in $\mathcal{T}_{N,d}$, there is a tree in $\mathcal{T}_{N,d-1}$ with a smaller $\mathcal{H}_2$ norm. Hence, the star $K_{1,N-1}$ has the smallest $\mathcal{H}_2$ norm of any tree with $N$ nodes.
\end{theorem}

\begin{proof}
By Lemma \ref{lem:Trlk}, $\mathcal{H}_2\left(P_{N,N-2,i}\right) < \mathcal{H}_2\left(P_N\right)$ (for any $1 \leq i \leq \lfloor\frac{N-2}{2}\rfloor$).

Let $4 \leq d \leq N-2$. Suppose $T \in \mathcal{T}_{N,d}$. Then by Theorem \ref{thm:P}, $\mathcal{H}_2(T) \geq \mathcal{H}_2\left(P_{N,d,\lfloor\frac{d}{2}\rfloor}\right)$. But by Lemma \ref{lem:Trlk}, $\mathcal{H}_2\left(P_{N,d,\lfloor\frac{d}{2}\rfloor}\right) > \mathcal{H}_2\left(P_{N,d-1,\lfloor\frac{d-1}{2}\rfloor}\right)$. Thus $\mathcal{H}_2\left(P_{N,d-1,\lfloor\frac{d-1}{2}\rfloor}\right) < \mathcal{H}_2(T)$.

Let $T \in \mathcal{T}_{N,3}$. Then by Theorem \ref{thm:d3}, $\mathcal{H}_2(T) \geq \mathcal{H}_2\left(D_{N,1,N-3}\right)$. But by Lemma \ref{lem:Trlk}, $\mathcal{H}_2\left(K_{1,N-1}\right) < \mathcal{H}_2\left(D_{N,1,N-3}\right)$. Thus $\mathcal{H}_2\left(K_{1,N-1}\right) < \mathcal{H}_2(T)$.
\end{proof}

\section{DISCUSSION}\label{sec:disc}

We were able to derive the results in Section \ref{sec:manipulate} by determining the effect that moving leaves has on effective resistances. With our well-defined definition of ``directed resistance'' for directed graphs, the same calculations can be made for directed trees as well. Thus our approach in this paper provides a constructive method for deriving a partial ordering of directed trees according to their $\mathcal{H}_2$ norms. Further details will appear in a future publication. Previous known results on the Wiener index of trees do not provide the same opportunity for the examination of directed trees.

Additionally, the manipulations we used to prove our results suggest how trees can be rearranged to improve their $\mathcal{H}_2$ norms in a decentralized fashion. In particular, we showed that for a non-star tree, the $\mathcal{H}_2$ norm can always be reduced either by moving a single node to somewhere else in the tree, or by moving a bouquet of nodes to an adjacent node.  These manipulations are ``local'' in the sense that nodes are moved only from a single location in the tree at a time, and the rest of the nodes in the tree are not required to take any additional action. The nodes that are chosen to be moved are always as far as possible from the center of a longest path.

Consider a large tree connecting many nodes, each of which has only local information about the tree. For example, suppose each node knows the graph between it and a fixed number of other nodes, as well as the degrees of each of these nodes. If the local neighborhood of node $i$ connects to the rest of the tree through a single other node $j$, and node $i$ is a leaf furthest from node $j$ within its local neighborhood, then $i$ is a candidate to be moved one node closer to $j$. Once such nodes identify themselves, they could move, form a new tree, and repeat the process.

A similar procedure could also be applied to situations in which there are constraints on the tree, such as maximum connectivities of nodes, or certain paths/structures that cannot be changed. 

Lemmas \ref{lem:cat}, \ref{lem:noncat} and \ref{lem:PvN} also provide insight into additional ordering within $\mathcal{T}_{N,d}$ that we have not presented here. A complete ordering of general trees is under investigation.


\bibliographystyle{IEEEtran} 
\bibliography{IEEEabrv,H2_Robustness}

\begin{thebibliography}{10}
\providecommand{\url}[1]{#1}
\csname url@rmstyle\endcsname
\providecommand{\newblock}{\relax}
\providecommand{\bibinfo}[2]{#2}
\providecommand\BIBentrySTDinterwordspacing{\spaceskip=0pt\relax}
\providecommand\BIBentryALTinterwordstretchfactor{4}
\providecommand\BIBentryALTinterwordspacing{\spaceskip=\fontdimen2\font plus
\BIBentryALTinterwordstretchfactor\fontdimen3\font minus
  \fontdimen4\font\relax}
\providecommand\BIBforeignlanguage[2]{{%
\expandafter\ifx\csname l@#1\endcsname\relax
\typeout{** WARNING: IEEEtran.bst: No hyphenation pattern has been}%
\typeout{** loaded for the language `#1'. Using the pattern for}%
\typeout{** the default language instead.}%
\else
\language=\csname l@#1\endcsname
\fi
#2}}

\bibitem{OlfatiSaber2004}
R.~Olfati-Saber and R.~Murray, ``Consensus problems in networks of agents with
  switching topology and time-delays,'' \emph{{IEEE} Trans. Automat. Contr.},
  vol.~49, no.~9, pp. 1520--1533, 2004.

\bibitem{Moreau2005}
L.~Moreau, ``Stability of multiagent systems with time-dependent communication
  links,'' \emph{{IEEE} Trans. Automat. Contr.}, vol.~50, no.~2, pp. 169--182,
  2005.

\bibitem{Ren2005}
W.~Ren, R.~Beard, and E.~Atkins, ``A survey of consensus problems in
  multi-agent coordination,'' in \emph{Proc. ACC}, 2005, pp. 1859--1864.

\bibitem{Blondel2005}
V.~Blondel, J.~Hendrickx, A.~Olshevsky, and J.~Tsitsiklis, ``Convergence in
  multiagent coordination, consensus, and flocking,'' in \emph{Proc. CDC-ECC},
  Seville, Spain, 2005, pp. 2996--3000.

\bibitem{Bamieh2008}
B.~Bamieh, M.~Jovanovic, P.~Mitra, and S.~Patterson, ``Effect of topological
  dimension on rigidity of vehicle formations: Fundamental limitations of local
  feedback,'' in \emph{Proc. CDC}, Cancun, Mexico, 2008, pp. 369--74.

\bibitem{OlfatiSaber2005}
R.~Olfati-Saber and J.~Shamma, ``Consensus filters for sensor networks and
  distributed sensor fusion,'' in \emph{Proc. CDC-ECC}, 2005, pp. 6698--6703.

\bibitem{Xiao2007}
L.~Xiao, S.~Boyd, and S.~Kim, ``Distributed average consensus with
  least-mean-square deviation,'' \emph{J. Parallel and Distributed Computing},
  vol.~67, no.~1, pp. 33--46, 2007.

\bibitem{Sumpter2008}
D.~Sumpter, J.~Krause, R.~James, I.~Couzin, and A.~Ward, ``Consensus decision
  making by fish,'' \emph{Curr. Bio.}, vol.~18, pp. 1773--1777, 2008.

\bibitem{Young2010}
G.~Young, L.~Scardovi, and N.~Leonard, ``Robustness of noisy consensus dynamics
  with directed communication,'' in \emph{Proc. ACC}, Baltimore, MD, 2010, pp.
  6312--6317.

\bibitem{Wu2007}
Z.~Wu, Z.~Guan, and X.~Wu, ``Consensus problem in multi-agent systems with
  physical position neighbourhood evolving network,'' \emph{Physica A Stat.
  Mech. Appl.}, vol. 379, no.~2, pp. 681--690, 2007.

\bibitem{Scardovi2009}
L.~Scardovi and N.~Leonard, ``Robustness of aggregation in networked dynamical
  systems,'' in \emph{Proc. ROBOCOMM}, Odense, Denmark, 2009, pp. 1--6.

\bibitem{Dobrynin2001}
A.~Dobrynin, R.~Entringer, and I.~Gutman, ``{Wiener} index of trees: Theory and
  applications,'' \emph{Acta Applicandae Mathematicae}, vol.~66, no.~3, pp.
  211--249, 2001.

\bibitem{Dong2006}
H.~Dong and X.~Guo, ``Ordering trees by their {Wiener} indices,'' \emph{MATCH
  Commun. Math. Comput. Chem.}, vol.~56, no.~3, pp. 527--540, 2006.

\bibitem{Deng2007}
H.-Y. Deng, ``The trees on $n \geq 9$ vertices with the first to seventeenth
  greatest {Wiener} indices are chemical trees,'' \emph{MATCH Commun. Math.
  Comput. Chem.}, vol.~57, no.~2, pp. 393--402, 2007.

\bibitem{Wang2008a}
S.~Wang and X.~Guo, ``Trees with extremal {Wiener} indices,'' \emph{MATCH
  Commun. Math. Comput. Chem.}, vol.~60, no.~2, pp. 609--622, 2008.

\bibitem{Mohar1991a}
B.~Mohar, ``The {Laplacian} spectrum of graphs,'' \emph{Graph theory,
  Combinatorics and Applications}, vol.~2, pp. 871--898, 1991.

\bibitem{Zelazo2009}
D.~Zelazo and M.~Mesbahi, ``{$\mathcal{H}_2$} performance of agreement protocol
  with noise: An edge based approach,'' in \emph{Proc. CDC}, Shanghai, China,
  2009, pp. 4747--4752.

\bibitem{Ghosh2008}
A.~Ghosh, S.~Boyd, and A.~Saberi, ``Minimizing effective resistance of a
  graph,'' \emph{SIAM Review}, vol.~50, no.~1, pp. 37--66, 2008.

\bibitem{Xiao2003}
W.~Xiao and I.~Gutman, ``Resistance distance and {Laplacian} spectrum,''
  \emph{Theor. Chem. Acc.}, vol. 110, no.~4, pp. 284--289, 2003.

\bibitem{Klein1993}
D.~Klein and M.~Randi\'{c}, ``Resistance distance,'' \emph{J. Math. Chem.},
  vol.~12, pp. 81--95, 1993.

\bibitem{Rouvray1987}
D.~Rouvray, ``The modeling of chemical phenomena using topological indices,''
  \emph{J. Comp. Chem.}, vol.~8, no.~4, pp. 470--480, 1987.

\bibitem{Entringer1994}
R.~Entringer, A.~Meir, J.~Moon, and L.~Sz\'{e}kely, ``On the {Wiener} index of
  trees from certain families,'' \emph{Australas. J. Combin.}, vol.~10, pp.
  211--224, 1994.

\bibitem{Simic2007}
S.~Simi\'{c} and B.~Zhou, ``Indices of trees with a prescribed diameter,''
  \emph{Applicable Analysis and Discrete Mathematics}, vol.~1, pp. 446--454,
  2007.

\bibitem{Grone1990}
R.~Grone and R.~Merris, ``Ordering trees by algebraic connectivity,''
  \emph{Graphs and Combinatorics}, vol.~6, no.~3, pp. 229--237, 1990.

\bibitem{Yuan2008}
X.~Yuan, J.~Shao, and L.~Zhang, ``The six classes of trees with the largest
  algebraic connectivity,'' \emph{Discrete Applied Mathematics}, vol. 156,
  no.~5, pp. 757--769, 2008.

\end{thebibliography}

\end{document}